\documentclass[12pt]{article}

\usepackage[paperwidth=240mm, paperheight=340mm]{geometry}

\usepackage{graphicx}
\usepackage{amssymb}
\usepackage{amsmath}
\usepackage{enumerate}
\usepackage{amsthm}
\usepackage{amsmath}
\usepackage{float}
\usepackage{lineno}
\makeatletter
\renewenvironment{proof}[1][\proofname]{%
   \par\pushQED{\qed}\normalfont%
   \topsep6\p@\@plus6\p@\relax
   \trivlist\item[\hskip\labelsep\bfseries#1\@addpunct{.}]%
   \ignorespaces
}{%
   \popQED\endtrivlist\@endpefalse
}
\makeatother

\newtheorem{theorem}{Theorem}
\newtheorem{proposition}[theorem]{Proposition}
 \numberwithin{theorem}{section}
 \newtheorem{corollary}[theorem]{Corollary}

\numberwithin{equation}{section}

\newtheorem{thmx}{Theorem}

\newcommand{\lL}{\mathsf{L}}
\newcommand{\rR}{\mathsf{R}}
\renewcommand{\P}{\mathbb{P}}
\newcommand{\E}{\mathbb{E}}
\newcommand{\R}{\mathbb{R}}

\newcommand{\mN}{\mathcal{N}}

\newcommand{\N}{\mathbb{N}}

\newcommand{\cF}{\mathcal F}
\newcommand{\cZ}{\mathcal Z}

\newcommand{\cW}{\mathcal W}
\newcommand{\eps}{\varepsilon}

 \newcommand{\no}{\noindent}

\begin{document}
\author{
Jieliang Hong\footnote{Department of Mathematics, University of British Columbia, Canada, E-mail: {\tt jlhong@math.ubc.ca} }
}
\title{Improved H\"{o}lder continuity near the boundary of one-dimensional super-Brownian motion}
%\affil{Department of Mathematics, The University of Brithish Columbia,1984 Mathematics Road, Vancouver, BC V6T1Z3, Canada E-mail address: jlhong@math.ubc.ca}
%\affil[$\dagger$]{Department of Mathematics, Pennsylvania State University,Pittsburgh, Pennsylvania 13593}
\date{\today}
\maketitle
\begin{abstract}
    We show that the local time of one-dimensional super-Brownian motion is locally $\gamma$-H\"older continuous near the boundary if $0<\gamma<3$ and fails to be locally $\gamma$-H\"older continuous if $\gamma>3$.
\end{abstract}

\section{Introduction}

Let $M_F=M_F(\R^d)$ be the space of finite measures on $(\R^d,\mathfrak{B}(\R^d))$ equipped with the topology of weak convergence of measures, and write $\mu(\phi)=\int \phi(x) \mu(dx)$ for $\mu \in M_F$. Let $(\Omega,\cF, (\cF_t)_{t\geq 0},P)$ be a filtered probability space. A super-Brownian motion $(X_t,\ t\geq 0)$ starting at $\mu \in M_F$ is a continuous $M_F$-valued strong $(\cF_t)_{t\geq 0}$-Markov process defined on $(\Omega,\cF,(\cF_t)_{t\geq 0},P)$ with $X_0=\mu$ a.s.. It is well known that super-Brownian motion is the solution to the following $\mathit{martingale\ problem}$ (see [Per02], II.5): For any $ \phi \in C_b^2(\R^d)$, 
\begin{equation} \label{e1.0}
X_t(\phi)=X_0(\phi)+M_t(\phi)+\int_0^t X_s(\frac{\Delta}{2}\phi) ds,\ \forall t\geq 0,
\end{equation} where $(M_t(\phi))_{t\geq 0}$ is a continuous $(\cF_t)_{t\geq 0}$-martingale such that $M_0(\phi)=0$ and  \[[M(\phi)]_t=\int_0^t X_s(\phi^2) ds,\ \forall t\geq 0.\]
The above martingale problem uniquely characterizes the law $\P_{X_0}$ of super-Brownian motion $X$, starting from $X_0 \in M_F$, on $C([0,\infty),M_F)$, the space of continuous functions from $[0,\infty)$ to $M_F$ furnished with the compact open topology. In particular, if we let $X_0$ be the Dirac mass $\delta_0$, then $\P_{\delta_0}$ denotes the law of super-Brownian motion $X$ starting from $\delta_0$.\\

Local times of superprocesses have been studied by many authors (cf. [Sug89], [BEP91], [AL92], [Kro93], [Mer06]). We recall that [Sug89] has proved that for $d\leq 3$, there exists a jointly lower semi-continuous local time $L_t^x$, which is monotone increasing in $t$ for all $x$, such that
\begin{equation} \label{e1.2}
\int_0^t X_s(\phi) ds =\int_{\R^d}  \phi(x) L_t^x dx, \text{ for all } t\geq 0 \text{ and  non-negative measurable }\phi.
\end{equation} 
Moreover, there is a version of the local time $L_t^x$ which is jointly continuous on the set of continuity points of $X_0 q_t(x)$, where $q_t(x)=\int_0^t p_s(x)ds$, $p_t(x)$ is the transition density of Brownian motion, and $X_0 q_t(x)=\int  q_t(y-x) X_0(dy)$ (see Theorem 3 in [Sug89]). Let the extinction time $\zeta$ of $X$ be defined as $\zeta=\zeta_X=\inf \{t\geq 0: X_t(1)=0\}$. We know that $\zeta<\infty$ a.s. (see Chp. II.5 in [Per02]). Then we have $L^x=L_{\infty}^x=L_\zeta^x$ is also lower semicontinuous. Note the set $\overline{\{x: L^x>0\}}$ is defined to be the range of super-Brownian motion $X$ (see [MP17]). Theorem 2.2 of [MP17] gives that for any $\eta>0$, with $\P_{\delta_0}$-probability one we have $L^x$ is $C^{(4-d)/2-\eta}$-H\"older continuous for $x$ away from $0$ if $d\leq 3$. When $d=1$, $L^x$ is globally continuous (see Proposition 3.1 in [Sug89]).\\

\noindent $\mathbf{Definition.}$ A function $f: \R \to \R$ is said to be locally $\gamma$-H\"older continuous at $x\in \R$, if there exist $\delta>0$ and $c>0$ such that
\[|f(x)-f(y)|\leq c|x-y|^\gamma, \ \forall y \text{ with } |y-x|<\delta.\]
We refer to $\gamma>0$ as the H\"older index and to $c > 0$ as the H\"older constant.\\

The problem studied in this paper was originally motivated by a heuristic calculation of the Hausdorff dimension, $d_f$, of the boundary of $\{x : L^x>0\}$ in [MP17]. With the following bounds given in Theorem 1.3 of [MP17],
\[\P_{\delta_0}(0<L^x\leq a) \leq Ca^{\alpha} \text{ for } a \text{ small,} \]
and an improved $\gamma$-H\"older continuity of $L^x$ for $x$ near its zero set, the two authors derived the upper bound $d_f\leq d-\alpha \gamma$ by a heuristic covering lemma in Section 1 of the same reference. Although these arguments were given for $d = 3$, they work in any dimension. As $d_f$ and $\alpha$ are
known from [MP17], one can reverse engineer and find the required $\gamma$. This leads to their
conjecture [private communication] that for any $\eta>0$, with $\P_{\delta_0}$-probability one
\begin{equation}
 x\to L^x\ \text{ is locally H\"older continuous of index } 4-d-\eta\ \text{ near the zero set of } L^x.
 \end{equation}
In [MP17] they reported that they can establish the above for $d = 3$ (and make the
argument for the upper bound on $d_f$ work). In this paper we confirm the above
conjecture for $d = 1$, as stated in Theorem 1.1 below. This result also gives us confidence on the validity of the $d = 2$ case, which remains an interesting open problem.\\

%These results go back to the study of the Hausdorff dimension of the boundary of super-Brownian motion in [MP17], where they conjecture that for $d\leq 3$,  The two authors claim that by using this improved modulus of continuity near the boundary, together with Theorem 1.3 of [MP17], they may get an upper bound that $dim(F) \leq d-(p-2)$ a.s., where $dim(F)$ denotes the Hausdorff dimension of the boundary of super-Brownian motion (see Chp. 1 of [MP17] for more discussions). For $d=3$, Theorem 2.2 of [MP17] implies that $x\to L^x$ is locally H\"older of index $1/2-\eta$ on $\{x\neq 0\}$. Mytnik and Perkins [private communication] improved this modulus to H\"older $1-\eta$ if one of the endpoints is in the zero set of $L^x$ (see Theorem 2.3 in [MP11] for a similar result for the density of $X$ if $d=1$). In this paper, we prove this conjecture with $d=1$ , whose proof will be given in Section 2, while the $d=2$ case remains as an open problem. The lower bound of the local time near the boundary goes back to the proof of Theorem 1.7 in [MP17], and we use the idea there to finish the .\\

\noindent  To state our main results, we first recall a result from Theorem 1.7 in [MP17].

\begin{thmx}([MP17])
If $d=1$ then $\P_{\delta_0}$-a.s. there are random variables $\lL<0<\rR$ such that 
\[\{x: L^x>0\}=(\lL,\rR).\]
\end{thmx}

\noindent As discussed above, we are interested in the decay rate of the local time $L^x$ on the boundary, i.e., at $\lL$ and $\rR$. \\

\begin{theorem}\label{t1}
Let $d=1$. If $0<\gamma<3$, then $\P_{\delta_0}$-a.s. the local time $L^x$ is locally $\gamma$-H\"older continuous at $\lL$ and $\rR$.
 \end{theorem}
 
\noindent This result will be proved in Section 2 and it is optimal in the sense of the following theorem, whose proof will be given in Section 3.

\begin{theorem}\label{t2}
Let $d=1$. For any $\gamma>3$, we have $\P_{\delta_0}$-a.s. that there is some $\delta(\gamma,\omega)>0$ such that $L^x\geq 2^{-\gamma/2} (\rR-x)^\gamma$ for all $\rR-\delta <x<\rR$.
 \end{theorem}
\noindent With the lower bound established above, we can get the following result immediately.
\begin{corollary}\label{t3}
Let $d=1$. If $\gamma>3$, then $\P_{\delta_0}$-a.s. the local time $L^x$ fails to be locally $\gamma$-H\"older continuous at $\lL$ and $\rR$.
 \end{corollary}
 
 \begin{proof}
By symmetry we may consider only $\rR$. For any $\gamma>3$, define $\gamma'=(3+\gamma)/2$ such that $3<\gamma'<\gamma$.  Then Theorem \ref{t2} would imply that $\P_{\delta_0}$-a.s. that there is some $\delta(\gamma',\omega)>0$ such that $L^x\geq 2^{-\gamma'/2} (\rR-x)^{\gamma'}$ for all $\rR-\delta <x<\rR$. For $\omega$ as above and
$c>0$, if $x<\rR$ is chosen close enough to $\rR$, then
\[L^x\geq 2^{-{\gamma'}/2}(\rR-x)^{\gamma'}> c(\rR-x)^\gamma,\]
and so the local $\gamma$-H\"older continuity at $\rR$ fails a.s..
\end{proof}

Now we continue to study the case under the canonical measure $\N_0$. $\N_{x_0}$ is a $\sigma$-finite measure on $C([0,\infty),M_F)$ which arises as the weak limit of $N P_{\delta_{x_0}/N}^N (X_\cdot^N \in \cdot)$ as $N \to \infty$, where $X_\cdot^N$ under $P_{\delta_{x_0}/N}^N$ is the approximating branching particle system starting from a single particle at $x_0$ (see Theorem II.7.3(a) in [Per02]). In this way it describes the contribution of a cluster from a single ancestor at $x_0$, and the super-Brownian motion is then obtained by a Poisson superposition of such clusters. In fact,  if we let
$\Xi=\sum_{i\in I}\delta_{\nu^i}$ be a Poisson point process on $C([0,\infty),M_F)$ with intensity $\N_{x_0}(d\nu)$, then

\begin{equation*}
X_t=\sum_{i\in I}\nu_t^i=\int \nu_t\ \Xi(d\nu),\ t>0, 
\end{equation*}
has the law, $\P_{\delta_{x_0}}$, of a super-Brownian motion $X$ starting from $\delta_{x_0}$. We refer the readers to Theorem II.7.3(c) in [Per02] for more details. The existence of the local time $L^x$ under $\N_{x_0}$ will follow from this decomposition and the existence under $\P_{\delta_{x_0}}$. Therefore the local time $L^x$ may be decomposed as 
\begin{align}\label{e1.21}
L^x=\sum_{i\in I}L^x(\nu^i)=\int L^x(\nu)\Xi(d\nu).
\end{align}
The continuity of local times $L^x$ under $\N_{x_0}$ is given in Theorem 1.2 of [Hong18]. We first give a version of Theorem A under the canonical measure.
\begin{theorem}\label{t4}
If $d=1$ then $\N_0$-a.e. there are random variables $\lL<0<\rR$ such that 
\[ \{x: L^x>0\}= (\lL,\rR).\]
\end{theorem}
%\no One can also expect that the local times under $\N_0$ behave in the same way near the boundary as that under $\P_{\delta_0}$
 \begin{theorem}\label{t5}
Theorem \ref{t1}, Theorem \ref{t2} and Corollary \ref{t3} hold if  $\P_{\delta_0}$ is replaced with $\N_0$.
\end{theorem}
\no The proofs of these analogous results under $\N_0$ will be given in Section 4.

%We first obtain a Tanaka formula for the local time in $d=1$, following the strategy in a recent paper for the local time in $d=3$ and $d=2$ (see Proposition 2.4 in [Hong18]). Then use Sugitani's result to get a similar Tanaka formula for the derivatives of the local time in $d=1$. Following the ideas from Theorem 4.1 and Corollary 4.2 of [MPS06] and Lemma 5.7 of [MP11], we get a similar result (see Lemma 2.3 below). The proof of Theorem 1.1 then follows with some recurring relations. 

\section*{Acknowledgements}
This work was done as part of the author's graduate studies at the University of British Columbia. I would like to express many thanks to my supervisor, Professor Edwin Perkins, for suggesting this problem and for showing me the ideas of the proof for the case $d=3$ in private conversations. I also thank anonymous referees for their careful reading of the manuscript and helpful comments.

\section{Upper bound of the local time near the boundary}

Let $g_x(y)=|y-x|$. Then $\frac{d^2}{dy^2} g_x(y)=2\delta_x(y)$ holds in the distributional sense and the martingale problem \eqref{e1.0} suggests the following result.
\begin{proposition}(Tanaka formula for d=1) 
Let $d=1$ and fix $x\neq 0$ in $\R^1$. Then we have $\P_{\delta_0}$-a.s. that
\begin{equation}\label{e2.7}
L_t^x+ |x| = X_t(g_x)-M_t(g_x), \ \forall t\geq 0,
\end{equation}
where $t\mapsto X_t(g_x)$ is continuous for $t\geq 0$ and $(M_t(g_x))_{t\geq 0}$ is a continuous $L^2$ martingale which is the stochastic integral with respect to the martingale measure associated with super-Brownian motion.
\end{proposition}
\begin{proof}
Let $(P_t)_{t\geq 0}$ be the Markov semigroup of one-dimensional Brownian motion. By cutoff arguments similar to those used in the proof of Proposition 2.4 in [Hong18], we may use the martingale problem \eqref{e1.0}  to see that for any $\eps>0$, with $\P_{\delta_0}$-probability one we have
\begin{equation}\label{e2.7.1}
 X_t(P_\varepsilon g_x)=P_\varepsilon g_x(0)+M_t(P_\varepsilon g_x)+\int_0^t X_s(\frac{\Delta}{2} P_\varepsilon g_x)ds, \ \forall t\geq 0.
\end{equation}
One can check that 
\begin{equation}\label{e0.1.0}
|P_\varepsilon g_x(y)-g_x(y)|\leq \varepsilon^{1/2},\ \forall x, y \in \R,
\end{equation}
and so it follows that
\begin{equation}\label{e0.1.1}
\Big|P_\varepsilon g_x(0)-|x|\Big| \to 0, \text{ as } \varepsilon \downarrow 0.
\end{equation}
Use \eqref{e0.1.0} again to see that for any $T>0$,
\begin{equation}\label{e0.2.2}
\sup_{t\leq T}\Big|X_t(P_\varepsilon g_x)-X_t(g_x)\Big|\leq \eps^{1/2} \sup_{t\leq T} X_t(1)  \to 0, \quad  \P_{\delta_0}-a.s.,
\end{equation}
  and 
\begin{equation}\label{e0.1.3}
 \E_{\delta_0}\Big[\Big(\sup_{t\leq T}\Big|M_t(P_\varepsilon g_x)-M_t(g_x)\Big|\Big)^2\Big] \leq 4 \E_{\delta_0}\Big[\int_0^T X_s\Big((P_\varepsilon g_x-g_x)^2\Big)ds \Big]\to 0.
\end{equation}
 The last inequality follows by Doob's inequality. Now for the convergence of last term on the right-hand side of \eqref{e2.7.1}, we apply integration by parts to get for any $\varepsilon>0$, $\frac{d^2}{dy^2} P_\varepsilon g_x(y)= 2 p_\varepsilon(y-x)=:2p_\varepsilon^x(y).$ Theorem 6.1 in [BEP91] gives us that as $\varepsilon \to 0$, 
 \begin{equation}\label{e0.2.3}
 \sup_{t\leq T} |\int_0^t X_s(p_{\varepsilon}^x) ds-L_t^x|\to 0,\quad \P_{\delta_0}-a.s.,
 \end{equation}
  and hence by taking an appropriate subsequence $\varepsilon_{n} \downarrow 0$, \eqref{e2.7} would follow immediately from \eqref{e2.7.1}, \eqref{e0.1.1}, \eqref{e0.2.2}, \eqref{e0.1.3} and \eqref{e0.2.3}. 
\end{proof}
Now we discuss the differentiability of $L_t^x$ in $d=1$ . We denote, by $D_{x} f(x)$ (resp. $D_{x}^{+}f(x),$ $D_{x}^{-}f(x)$), the derivative (resp. right derivative, left derivative) of $f(x)$. Then we have the following result from Theorem 4 of  [Sug89].
 \begin{thmx}([Sug89])
Let $d=1$ and $X_0=\mu \in M_F(\R)$. Then the following (i) and (ii) hold with $\P_\mu$-probability one.
\begin{enumerate}[(i)]
\item $Z(t,x)=L_t^x-\E_\mu(L_t^x)$ is differentiable with respect to $x$, $\forall t\geq 0$;
\item $D_x Z(t,x)$ is jointly continuous in $t\geq 0$ and $x\in \R$, and we have
\begin{align}
D_x^{+} \E_\mu(L_t^x)-D_x^{-} \E_\mu(L_t^x)=-2\mu(\{x\}), \ t>0, x\in \R.
\end{align}
\end{enumerate}
In particular, if we let $H=\{x\in \R: \mu(\{x\})=0\}$, then $D_x \E_\mu(L_t^x)$ is jointly continuous on $[0,\infty)\times H$ and so with $\P_\mu$-probability one  we have  $L_t^x$ is differentiable with respect to $x$ on $H$ and $D_x L_t^x$ is jointly continuous on $[0,\infty)\times H$.
\end{thmx}

So for the case $X_0=\delta_0$, we know from the above theorem that $L_t^x$ is continuously differentiable on $\{x\neq 0\}$. Let $sgn(x)=x/|x|$ for $x\neq 0$ and $sgn(0)=0$. Then $D_y g_x(y)=sgn(y-x)$ for $y\neq x$ and we have the following Tanaka formula for $D_x L_t^x$.
\begin{proposition} \label{p2}
Let $d=1$ and fix $x\neq 0$ in $\R^1$. Then we have $\P_{\delta_0}$-a.s. that
\begin{equation}\label{e2.00}
D_x L_t^x=- sgn(x) + X_t(sgn(x-\cdot))-M_t(sgn(x-\cdot)), \ \forall t\geq 0.
\end{equation}
\end{proposition}

\begin{proof}
Fix any $x\neq 0$ and any $t\geq 0$. Choose some positive sequence $\{h_n\}_{n\geq 1}$ such that $h_n \downarrow 0$. Then use \eqref{e2.7} to see that with $\P_{\delta_0}$-probability one,
\begin{align}\label{e4}
\frac{1}{h_n}(L_t^{x+h_n}-L_t^x)+\frac{1}{h_n}(|x+h_n|-|x|)=\frac{1}{h_n}(X_t(g_{x+h_n})-X_t(g_{x}))-\frac{1}{h_n}(M_t(g_{x+h_n})-M_t(g_{x})).
\end{align}
By Theorem B, we conclude that the left hand side converges a.s. to $D_x L_t^x+ sgn(x)$ as $h_n \downarrow 0$. For the right hand side, first note that for all $x,y \in \R$, we have $|(|x+h-y|-|x-y|)/h| \leq 1$ for all $h>0$. Then bounded convergence theorem implies as $h_n \downarrow 0$,
\[\frac{1}{h_n}(X_t(g_{x+h_n})-X_t(g_{x_n}))=\int \frac{1}{h_n}(|x+h_n-y|-|x-y|) X_t(dy) \to \int sgn(x-y) X_t(dy),\] and
\begin{align*}
&\E_{\delta_0}\bigg[\Big(\frac{1}{h_n}(M_t(g_{x+h_n})-M_t(g_{x}))-M_t(sgn(x-\cdot))\Big)^2\bigg]\\
\leq& \E_{\delta_0}\bigg[\int_0^t  \int \left(\frac{1}{h_n}(|x+h_n-y|-|x-y|)-sgn(x-y)\right)^2 X_s(dy) ds\bigg] \\
=& \int_0^t ds \int p_s(y)\left(\frac{1}{h_n}(|x+h_n-y|-|x-y|)-sgn(x-y)\right)^2 dy \to 0.
\end{align*}
In the last equality we use $\E_{\delta_0} X_t(dy)=p_t(y) dy$ from Lemma 2.2 of [KS88]. So every term, except the last term on the right-hand side, in \eqref{e4} converges a.s. and hence the last term converges a.s. as well. Note we have shown that it converges in $L^2$ to $M_t(sgn(x-\cdot))$. Then it follows that the last term converges a.s. to $M_t(sgn(x-\cdot))$ and so \eqref{e2.00} for any fix $t\geq 0$ follows from \eqref{e4} .\\

Now take countable union of null sets to see that with $\P_{\delta_0}$-probability one, we have \eqref{e2.00} holds for all rational $t\geq 0$. Note by Theorem B we have $t\mapsto D_x L_t^x$ is continuous for all $t\geq 0$ $\P_{\delta_0}$-a.s.. For the right-hand side terms of \eqref{e2.00}, since $X_t(\{x\})=0$ for all $t\geq 0$ $\P_{\delta_0}$-a.s., the weak continuity of $t\mapsto X_t$ for all $t\geq 0$ would give us the continuity of $t\mapsto X_t(sgn(x-\cdot))$ for all $t\geq 0$. Next since $sgn(x-\cdot)$ is a bounded function and $M_t(sgn(x-\cdot))=\int_0^t \int sgn(x-y) M(dy ds)$ is an integral with respect to the martingale measure, it follows immediately that $t\mapsto M_t(sgn(x-\cdot))$ is continuous for all $t\geq 0$. Therefore we can upgrade the rational $t\geq 0$ to all $t\geq 0$ and the proof is complete.
\end{proof}

Now we will turn to the proof of Theorem \ref{t1}. By symmetry we can consider the case $x>0$. Since $X_t(1)=0$ for $t=\zeta$, $\P_{\delta_0}$-a.s., we use Proposition \ref{p2} with $t=\zeta$ to see that for any $x>0$, with $\P_{\delta_0}$-probability one we have
\[
L'(x):=D_x L^x =-1-\int_0^\infty \int sgn(x-z) M(dz ds).
\]
Define $N_t^{x,y}=\int_0^t \int (sgn(y-z)- sgn(x-z)) M(dz ds)$ for $x,y >0$ and $t\geq 0$. Then we have 
\begin{align}\label{e0.1.2}
L'(x)-L'(y)=N_\infty^{x,y}=\int_0^\infty \int (sgn(y-z)- sgn(x-z)) M(dz ds),
\end{align}
and its quadratic variation is 
\begin{align}\label{e0.0.1}
[N^{x,y}]_\infty=&\int_0^\infty  \int  (sgn(y-z)-sgn(x-z))^2 X_s(dz) ds\nonumber\\
=& \int (sgn(y-z)-sgn(x-z))^2 L^z dz =4\ \left|\int_x^y L^z dz\right|.
\end{align}
The second equality is by \eqref{e1.2} and the last follows since $(sgn(y-z)-sgn(x-z))^2 \equiv 4$ for $z$ between $x$ and $y$, and $\equiv 0$ otherwise.\\

The following theorem, which is a generalization of Theorem 4.1 of [MPS06], carries out the main bootstrap idea we use to prove Theorem \ref{t1}: we start from a lower order of H\"{o}lder continuity, say $\xi_0$, of the local time $L^x$ and then upgrade to a higher order of H\"{o}lder continuity $\xi_1\approx (3+\xi_0)/{2}$. By iterating we can reach the highest possible order $3$.\begin{theorem}\label{l1.4}
Let $Z_N$ be the random set $[\rR- 2^{-N}, \rR] \cap (0,\infty)$ for any positive integer $N\geq 1$, where $\rR$ is the r.v. from Theorem A. Assume $\xi_0 \in (0,3)$ satisfies 
\begin{align}
&\exists 1\leq N_{\xi_0}(\omega)<\infty \ a.s. \text{ such that } \forall N\geq N_{\xi_0}(\omega), x\in Z_N,\nonumber\\
&\forall |y-x|\leq 2^{-N} \Rightarrow |L^x-L^y|\leq 2^{-\xi_0 N}. \label{e0.102}
\end{align}
Then for all $0<\xi_1<(3+\xi_0)/{2}$,
\begin{align}
&\exists 1\leq N_{\xi_1}(\omega)<\infty  \ a.s. \text{ such that } \forall N\geq N_{\xi_1}(\omega), x\in Z_N,\nonumber\\
&\forall |y-x|\leq 2^{-N} \Rightarrow |L^x-L^y|\leq 2^{-\xi_1 N}. \label{e0.103}
\end{align}
\end{theorem}
\begin{proof}
Note that $\rR\in Z_N$ for all $N\geq 1$. By \eqref{e0.102}, we have
\begin{align}\label{e0.101}
|L^z|=|L^z-L^{\rR}|\leq 2^{-\xi_0 (N-1)}, \text{ if } z \in Z_{N-1},\ N\geq N_{\xi_0}+1. 
\end{align}
Let $N\geq N_{\xi_0}+1$. For $x \in Z_N$ and $|y-x|\leq 2^{-N}$, we have $y\in Z_{N-1}$ and $z\in Z_{N-1}$ for any $z$ between $x$ and $y$. Therefore \eqref{e0.0.1} implies
\begin{align}\label{e1}
[N^{x,y}]_\infty= 4\ |\int_x^y L^z dz| \leq 4\cdot 2^{- \xi_0 (N-1)} |y-x| \leq 2^5 \cdot 2^{-\xi_0 N} |y-x|, 
\end{align}
the first inequality by \eqref{e0.101} with $z\in Z_{N-1}$.\\

Pick $1/4<\eta<1/2$ such that 
\begin{align}\label{ea.1}
\eta(1+\xi_0)+1>\xi_1.
\end{align}
By using the Dubins-Schwarz theorem (see [RY94], Theorem V1.6 and V1.7), with an enlargement of the underlying probability space, we can construct some Brownian motion $(B(t), t\geq 0)$ in $\R$ such that $L'(x)-L'(y)=N_\infty^{x,y}=B([N^{x,y}]_\infty)$. So for any $N\in \N$, we have
\begin{align}\label{e0}
&\P_{\delta_0}(|L'(x)-L'(y)|\geq 2^{5} \cdot 2^{-\eta \xi_0 N} |y-x|^\eta, \ x\in Z_N, \ |y-x|\leq 2^{-N}, \ N\geq N_{\xi_0}+1) \nonumber\\
\leq& P(\sup_{s\leq 2^5 \cdot 2^{-\xi_0 N} |y-x|}   |B(s)| \geq 2^5 \cdot 2^{-\eta \xi_0 N} |y-x|^\eta )\nonumber\  \text{ (by \eqref{e1}) }\\
\leq & 2 \exp(-2^5 \cdot 2^{\xi_0 N(1-2\eta)} |y-x|^{2\eta-1} ).
\end{align}
For $k\geq N$, define \[M_{k,N}=\max\Big\{\ |L'(\rR-\frac{i+1}{2^k})-L'(\rR-\frac{i}{2^k})|: \ 0\leq i\leq 2^{k-N}\Big\},\] and 
\[A_N=\Big\{\omega: \exists\ k\geq N\ s.t.\ M_{k,N} \geq 2^5 \cdot 2^{-\eta \xi_0 N} 2^{-\eta k}, N\geq N_{\xi_0}+1  \Big\}.\]  Note for each $0\leq i\leq 2^{k-N}$, we have $\rR-i 2^{-k} \in Z_N$. Let $x=\rR-i 2^{-k}$ and $y=\rR-(i+1) 2^{-k}$ in \eqref{e0} to get
\begin{align}
&\P_{\delta_0}(|L'(\rR-\frac{i}{2^k})-L'(\rR-\frac{i+1}{2^k})|\geq 2^{5} \cdot 2^{-\eta \xi_0 N} 2^{-\eta k}, k\geq N\geq N_{\xi_0}+1) \\
 \leq& 2 \exp(-2^5 \cdot 2^{\xi_0 N(1-2\eta)} 2^{k(1-2\eta)} ), \nonumber
\end{align}
and hence
\[\P_{\delta_0}\Big(\bigcup_{N'=N}^\infty A_{N'}\Big)\leq \sum_{N'=N}^\infty \sum_{k=N'}^\infty (2^{k-N'}+1) \cdot 2 \exp(-2^5 \cdot 2^{\xi_0 N'(1-2\eta)} 2^{k(1-2\eta)} ) \leq c_0 \exp(-c_1 2^{ N(1+\xi_0)(1-2\eta)})\] for some constants $c_0,c_1>0$.
Let \[N_1=\min\{N \in \N: \omega \in \bigcap_{N'=N}^\infty A_{N'}^c\}.\] The above implies 
\[\P_{\delta_0}(N_1>N)=\P_{\delta_0}\Big(\bigcup_{N'=N}^\infty A_{N'}\Big) \leq c_0 \exp(-c_1 2^{ N(1+\xi_0)(1-2\eta)}),\]
and so $N_1$ is an a.s. finite random variable. Define 
\begin{align}\label{e0def}
N_{\xi_1}= N_1 \vee (N_{\xi_0}+1) \vee \frac{12}{\eta(1+\xi_0)+1-\xi_1}\vee 1,
\end{align}  where the third one is well defined by \eqref{ea.1}.
For all $N\geq N_{\xi_1}$, $k\geq N$, $x\in Z_N$ and $|y-x|\leq 2^{-N}$, let $x_k=\rR-\lfloor 2^k (\rR-x) \rfloor 2^{-k} \downarrow x$ and $y_k=\rR-\lfloor 2^k (\rR-y) \rfloor 2^{-k} \downarrow y$. Then $|x_k-x_{k+1}|\leq 2^{-(k+1)}$ and $|y_k-y_{k+1}|\leq 2^{-(k+1)}$. Note $x_N, y_N\in \{\rR,\rR-2^{-N}, \rR-2^{1-N}\}$ and $|x_N-y_N|\leq 2^{-N}$ since $|y-x|\leq 2^{-N}$. The continuity of $L'(x)$ gives
\[L'(x)=-L'(x_N)+\sum_{k=N}^\infty \big(L'(x_k)-L'(x_{k+1})\big),\]
and
\[L'(y)=-L'(y_N)+\sum_{k=N}^\infty \big(L'(y_k)-L'(y_{k+1})\big).\]
So
\begin{align}
&|L'(x)-L'(y)| \\
\leq& |L'(x_N)-L'(y_N)|+\sum_{k=N+1}^\infty \Big( |L'(x_k)-L'(x_{k+1})|+|L'(y_k)-L'(y_{k+1})| \Big)\nonumber\\
\leq& M_{N,N}+\sum_{k=N}^\infty 2 M_{k+1,N}\leq 2^5 \cdot 2^{-\eta \xi_0 N} 2^{-\eta N}+2\sum_{k=N}^\infty 2^5 \cdot 2^{-\eta \xi_0 N} 2^{-\eta (k+1)}\nonumber \\
\leq& 2^{10} \cdot 2^{-\eta N (\xi_0+1)},\nonumber
\end{align}
where we have used the definitions of $M_{k,N}$ and $A_{N}$ and $N\ge N_{\xi_1}\geq N_1 \vee (N_{\xi_0}+1)$ by \eqref{e0def} in the third line.
Let $x=z \in Z_N$ and $y=\rR$ in above. Then use $L'(\rR)=0$ to see that 
\begin{align}\label{e3}
|L'(z)|\leq 2^{10} \cdot 2^{-N\eta (1+\xi_0)},\ \forall z\in Z_N, N\geq N_{\xi_1}.
\end{align}
Let $N\geq N_{\xi_1}+1$. For $x \in Z_N$ and $|y-x|\leq 2^{-N}$, we have $y\in Z_{N-1}$ and $z\in Z_{N-1}$ for any $z$ between $x$ and $y$. Use \eqref{e3} to get
\[|L(y)-L(x)|=|L'(z)||y-x|\leq 2^{10} \cdot 2^{-(N-1) \eta (1+\xi_0)} 2^{-N} \leq 2^{-\xi_1 N},\] the last by $N> N_{\xi_1}>12/(\eta(1+\xi_0)+1-\xi_1)$ and \eqref{ea.1}.
\end{proof}
Theorem \ref{t1} follows from the following corollary of the above result.
\begin{corollary}\label{c2.4}
Let $\gamma \in (0,3)$. Then $\P_{\delta_0}$-a.s. there is a random variable $\delta(\gamma, \omega)>0$ such that for any $0<\rR-x<\delta$, we have $L^x \leq 2^\gamma (\rR-x)^\gamma$.
\end{corollary}
\begin{proof}
By Theorem 2.2 in [MP17], for any $0<\xi_0<1$, with $\P_{\delta_0}$-probability one, there is some $0<\rho(\omega)\leq 1$ such that 
\begin{equation}\label{e0.0.3}
|L^y-L^x|<|y-x|^{\xi_0}, \ \text{ for } x,y>0 \text{ with }|y-x|<\rho.
\end{equation}
Note we may set $\varepsilon_0=0$ in Theorem 2.2 of [MP17] due to the global continuity of $L^x$ in $d=1$. Pick $\xi_0=1/2$, then \eqref{e0.102} in Theorem \ref{l1.4} holds for $N\geq N_{\xi_0}(\omega)=1 \vee \log_2 (\rho(\omega)^{-1} ) $. Inductively, define $\xi_{n+1}=\frac{1}{2}(3+\xi_n)(1-\frac{1}{n+3})$ so that $\xi_{n+1} \uparrow 3$. Pick $n_0$ such that $\xi_{n_0}\geq \gamma > \xi_{n_0-1}$. Apply Theorem \ref{l1.4} inductively $n_0$ times to get \eqref{e0.102} for $\xi_0=\xi_{n_0-1}$ and hence, \eqref{e0.103} with $\xi_1=\xi_{n_0}$.\\

Consider $0<\rR-x \leq 2^{-N_{\xi_{n_0}}}$. Choose $N\geq N_{\xi_{n_0}}$ such that $2^{-(N+1)}<\rR-x \leq 2^{-N}.$ Then $x\in Z_N$ and \eqref{e0.103} with $\xi_1=\xi_{n_0}$ implies
\begin{align}\label{e0.2}
|L^x|=|L^x-L^\rR|\leq 2^{-N\xi_{n_0}} \leq  2^{-N\gamma} \leq \left(2(\rR-x)\right)^\gamma= 2^\gamma (\rR-x)^\gamma.
\end{align}
The proof is completed by choosing $\delta=2^{-N_{\xi_{n_0}}}>0$.
\end{proof}

\section{Lower bound of the local time near the boundary}

\begin{proof}[Proof of Theorem \ref{t2}]
The proof of the lower bound on the local time near the boundary requires an application of Dynkin's exit measures of super-Brownian motion $X$. The exit measure of $X$ from an open set $G$ under $\P_{X_0}$ is denoted by $X_G$ (see Chp. V of [Leg99a] for the construction of the exit measure). Intuitively $X_{G}$ is a random finite measure supported on $\partial G$,  which corresponds to the mass started at $X_0$ which is stopped at the instant it leaves $G$. 
The Laplace functional of $X_G$ is given by 
\begin{align}\label{LFEM}
\E_{X_0}(\exp(-X_G(g))=\exp\Bigl(-\int U^g(x)X_0(dx)\Bigr),
\end{align}
where $g:\partial G\to[0,\infty)$ is continuous and $U^g\ge 0$ is the unique continuous function on $\overline G$ which is $C^2$ on $G$ and solves
\begin{equation}\label{EMpde}
\Delta U^g=(U^g)^2\text{ on }G,\quad U^g=g\text{ on }\partial G.
\end{equation}
Now we work with a one-dimensional super-Brownian motion $X$ with initial condition
$y_0 \delta_0$. For $r > 0$ we let $Y_r \delta_r$ denote the exit measure $X_{(-\infty, r)}$ from $(-\infty, r)$
and set $Y_0=y_0$. Then Proposition 4.1 of [MP17] implies under $\P_{y_0\delta_0}$ there is a cadlag version of $Y$ which is a stable continuous state branching process (SCSBP) starting at $y_0$ with parameter $3/2$, and so is an $(\cF_r^Y)_{r\geq 0}$-martingale with $\cF_r^Y=\sigma(Y_s, s\leq r)$ (see Section II.1 of [Leg99a] for the definition of (SCSBP)). In particular (4.6) in [MP17] gives
\[\E_{y_0 \delta_0}(\exp(-\lambda Y_r))=\exp(-6 y_0 (r+\sqrt{6/\lambda})^{-2}),\quad \forall \lambda\geq 0, r\geq 0.\]
Let $\lambda \uparrow \infty$, we have 
\begin{align}\label{e2.01}
\P_{y_0 \delta_0}(Y_r=0)=\exp(-6y_0 r^{-2}), \quad \forall r\geq 0.
\end{align}

Let $R_n= \inf\{r\geq 0: Y_r\leq 2^{-n}\} \uparrow \rR= \inf\{r\geq 0: Y_r=0\}$ as $n\to \infty$. Note the $\rR$ defined here will give the same $\rR$ in Theorem A. By repeating the arguments in the proof of Theorem 1.7 in [MP17], for any $\beta>3/2$, we have 
\begin{align}\label{e2.0}
\text{w.p.1} \ \exists N_0(\omega)<\infty, \text{ so that } \inf_{0<x<R_n} L^x >2^{-n\beta}, \ \forall n>N_0.
\end{align}
Note again we may set $\varepsilon_0=0$ in Theorem 2.2 of [MP17] due to the global continuity of $L^x$ in $d=1$ to get the above. The definition of $R_n$ implies $Y(R_n)=2^{-n}, \P_{\delta_0}$-a.s. as $Y$ is a SCSBP and hence it only has positive jumps, i.e. it is spectrally positive (see [CLB09]). So for any $0<\xi<1/2$, recalling that the non-negative martingale $Y$ stops at $0$ when it hits $0$ at time $\rR$, we see that
\begin{align*}
\P_{\delta_0} (|R_n-\rR|>(2^{-n})^\xi)= &\P_{\delta_0} (\rR>R_n+(2^{-n})^\xi)\leq \P_{\delta_0} (Y_{{R_n}+2^{-n\xi}}>0)\\
=&\E_{\delta_0} (\P_{\delta_0}(Y_{{R_n}+2^{-n\xi}}>0|\cF_{R_n}^Y))= \E_{\delta_0} (\P_{Y_{R_n} \delta_{0}}(Y_{2^{-n\xi}}>0))\\
=&\E_{\delta_0} (1-\exp(-6Y_{R_n} 2^{2n\xi}))\\
\leq& \E_{\delta_0}(6Y_{R_n} 2^{2n\xi})=6(\frac{1}{2^n})^{1-2\xi},
\end{align*}
where the second line holds by the strong Markov property of $Y$, and the third line uses \eqref{e2.01}. By Borel-Cantelli Lemma, w.p.1 there is some $N_1(\omega)<\infty$ such that 
\begin{align}\label{e3.1}
|R_n-\rR|\leq (\frac{1}{2^n})^\xi, \ \forall n\geq N_1.
\end{align}
For any fixed $\gamma>3$, pick $0<\xi<1/2$ such that $\gamma \xi >3/2$. Let $\beta=\gamma \xi>3/2$ in \eqref{e2.0} and define $N(\omega)=N_0(\omega) \vee N_1(\omega) <\infty$. Then it follows from \eqref{e3.1} that
\begin{align}\label{e3.2}
|R_n-\rR|^\gamma \leq  (\frac{1}{2^n})^{ \gamma\xi}, \ \forall n\geq N\geq N_1.
\end{align}
 For all $R_N\leq x < \rR$, there is some $n\geq N$ such that $R_n\leq x< R_{n+1}$. Now use \eqref{e2.0} with $n\geq N\geq (N_0 \vee N_1)$ to get 
\begin{align*}
|L^x-L^\rR|=L^x&\geq \inf_{0<y<R_{n+1}} L^y>2^{-\gamma \xi (n+1)} \geq 2^{-\gamma/2} (\frac{1}{2^n})^{\gamma \xi}\\
 &\geq  2^{-\gamma/2} |R_n-\rR|^\gamma \geq 2^{-\gamma/2} |x-\rR|^\gamma,
\end{align*}
where the second last inequality is by \eqref{e3.2}. The proof is completed by choosing $\delta=\rR-R_N>0$.
\end{proof}
%\noindent Now we continue to the 

\section{The case under the canonical measure}
In this paper we use Le Gall's Brownian snake approach to study super-Brownian motion under the canonical measure. 
Define $\cW=\cup_{t\ge 0} C([0,t],\R^d)$, equipped with the metric given in Chp IV.1 of [Leg99a], and denote by $\zeta(w)=t$ the lifetime of $w\in C([0,t],\R^d)\subset\cW$. The Brownian snake $W=(W_t, t\geq 0)$ constructed in Ch. IV of [Leg99a] is a $\cW$-valued continuous strong Markov process and we denote by $\N_{x_0}$ the excursion measure of $W$ away from the trivial path ${x_0}$ for $x_0\in\R^d$ with zero lifetime. The law of $X=X(W)$ under $\N_{x_0}$, constructed in Theorem IV.4 of [Leg99a],  is the canonical measure of super-Brownian motion described in the introduction (also denoted by $\N_{x_0}$). For our purpose it suffices to note that if $\Xi=\sum_{i\in I}\delta_{W_i}$ is a Poisson point process on the space of continuous $\cW$-valued paths with intensity $\N_{x_0}(dW)$, then
\begin{equation*}
X_t(W)=\sum_{i\in I}X_t(W_i)=\int X_t(W)\Xi(dW),\ t>0, 
\end{equation*}
has the law, $\P_{\delta_{x_0}}$, of a super-Brownian motion $X$ starting from $\delta_{x_0}$. Compared to \eqref{e1.21}, (2.19) of [MP17] implies that the local time $L^x$ may also be decomposed as 
\begin{equation}\label{e0.1}
L^x(W)=\sum_{i\in I}L^x(W_i)=\int L^x(W)\Xi(dW).
\end{equation}
Under the excursion measure $\N_{x_0}$, let $\sigma(W)=\inf\{t\geq 0: \zeta_t=0\}>0$ be the length of the excursion path where $\zeta_t=\zeta(W_t)$ is the life time of $W_t$ and $\hat W_t=W_t(\zeta_t)$ be the ``tip'' of the snake at time $t$. Then (2.20) of [MP17] implies that for any measurable function $\phi\geq 0$ ,
\begin{align}\label{e1.31}
\int_0^\infty X_s(\phi) ds=\int L^x \phi(x) dx=\int_0^\sigma \phi(\hat W_s) ds.
\end{align}

\begin{proof}[Proof of Theorem \ref{t4}]
Let $\rR=\sup \{x\geq 0: L^x>0\}$ and $\lL=\inf \{x\leq 0: L^x>0\}$. First we show that $L^0>0$, $\N_0$-a.e., and then by Theorem 1.2 of [Hong18], the continuity of local times under $\N_0$ in $d=1$ would imply that $\lL<0<\rR$, $\N_0$-a.e..

Define the occupation measure $\cZ$ by $\cZ(A)=\int_0^\sigma 1_A(\hat W_s) ds$ for all Borel measurable set $A$ on $\R$. Then \eqref{e1.31} implies that under $\N_{x_0}$, the local time $L^x$ coincides with the density function of the occupation measure $\cZ$, which we denote by $L^x(\cZ)$. By the Palm measure formula for $\cZ$ (see Proposition 16.2.1 of [Leg99b]) with $F(y,\cZ)=\exp(-\lambda L^0(\cZ))$ for any $\lambda>0$, we see that
\begin{align}\label{e1.5}
\N_0\Big(\cZ(1) 1(L^0=0)\Big)=&\lim_{\lambda \to \infty} \N_0\Big(\cZ(1) \exp(-\lambda L^0(\cZ))\Big)\\
=&\lim_{\lambda \to \infty} \int_0^\infty da \int P_0^a(dw) E^{(w)}\Big(\exp(-\lambda \int L^0(\cZ(\omega)) \mN(dtd\omega)\Big)\nonumber\\
=&\lim_{\lambda \to \infty} \int_0^\infty da \int P_0^a(dw)\exp\Big(-\int_0^{\zeta(w)}  \N_{w(t)}\big(1-\exp(-\lambda L^0)\big) dt\Big),\nonumber
\end{align}
where $P_0^a$ is the law of Brownian motion in $\R$ started at $0$ and stopped at time $a$ and for each $w$ under $P_0^a$, the probability measure $P^{(w)}$ is defined on an auxiliary probability space and such that under $P^{(w)}$, $\mN(dtd\omega)$ is a Poisson point measure with intensity $1_{[0,\zeta(w)]}(t) dt \N_{w(t)}(d\omega)$. Note here we have taken our branching rate for $X$ to be one and so our constants will differ from those in [Leg99b]. For each $w$ under $P_0^a$, we have $\zeta(w)=a$. Therefore the left-hand side of \eqref{e1.5} is equal to
\begin{align*}
\int_0^\infty da \int P_0^a(dw) \exp\Big(-\int_0^{a}  \N_{w(t)}\big(L^0>0\big) dt\Big) =& \int_0^\infty da \int P_0^a(dw) \exp\Big(-\int_0^{a} \frac{6}{|w(t)|^2} dt\Big),
\end{align*}
the last by (2.12) of [MP17]. By Levy's modulus of continuity, we have $\int_0^{a} 6/|w(t)|^2 dt=\infty$, $P_0^a$-a.s. for each $a>0$ and hence the above implies $\N_0\Big(\cZ(1) 1(L^0=0)\Big)=0$. Since $\cZ(1)=\sigma>0$, $\N_0$-a.e., we have 
\begin{align}
L^0>0,\quad \N_0-a.e..
\end{align}

%Simply observe that, for every $\lambda>0$, the process $(\lambda Z_t, t\geq 0)$ is a $(\xi, 2\lambda u^2)$-superprocess if $( Z_t, t\geq 0)$ is a $(\xi, 2 u^2)$-superprocess
Now we will show that $L^x$ is strictly positive on $(\lL,\rR)$. Fix $\varepsilon>0$ and let $L=(L^x, x>\varepsilon)$. Note that $\rR \leq \varepsilon$ implies $L^x\equiv 0$ for all $x>\varepsilon$ by definition. Then the canonical decomposition \eqref{e0.1} implies that under $\P_{\delta_0}$, $(L,N_\varepsilon)$ is equal in law to $(\sum_{i=1}^{N_\varepsilon} L_i,N_\varepsilon)$, where $N_\varepsilon$ is a Poisson random variable with parameter $\N_0(\rR>\varepsilon)<\infty$ and given $N_\varepsilon$, $(L_i=(L_i^x, x>\varepsilon))_{i \in \N}$ are i.i.d. with law $\N_0(L\in \cdot\big| \rR>\varepsilon)$. Theorem A implies that 
\begin{align*}
0&=\P_{\delta_0}(N_\varepsilon=1; \exists \varepsilon<x<\rR, L^x=0)=\P_{\delta_0}(N_\varepsilon=1) \N_0( \exists \varepsilon<x<\rR, L^x=0\big| \rR>\varepsilon). 
\end{align*}
Therefore we have $\N_0( \exists \varepsilon<x<\rR, L^x=0; \rR>\varepsilon)=0$ for all $\varepsilon>0$. Let $\varepsilon \downarrow 0$ to see that $\N_0( \exists 0<x<\rR, L^x=0; \rR>0)=0$. Since $\rR>0$, $\N_0$-a.e., we have $L^x>0, \forall 0<x<\rR$, $\N_0$-a.e.. Use symmetry to conclude for $\lL$.
\end{proof}

\begin{proof}[Proof of Theorem \ref{t5}]
Fix $\varepsilon>0$ and let $L=(L^x, x>\varepsilon)$. Use the same canonical decomposition above to see that under $\P_{\delta_0}$, $(L,N_\varepsilon)$ is equal in law to $(\sum_{i=1}^{N_\varepsilon} L_i,N_\varepsilon)$, where $N_\varepsilon$ and $(L_i=(L_i^x, x>\varepsilon))_{i \in \N}$ are as above. For any $\gamma \in (0,3)$,  use Corollary \ref{c2.4} to see that
\begin{align*}
0&=\P_{\delta_0}(N_\varepsilon=1; \exists x_n> \varepsilon,\  x_n \uparrow \rR, \ s.t.\ L^{x_n}>2^3 (\rR-x_n)^\gamma \text{ i.o.})\\
&=\P_{\delta_0}(N_\varepsilon=1) \N_0( \exists x_n> \varepsilon,\  x_n \uparrow \rR, \ s.t.\ L^{x_n}>2^3 (\rR-x_n)^\gamma \text{ i.o.} \big|\rR>\varepsilon),
\end{align*}
where i.o. represents infinitely often. Therefore we have $\N_0( \exists x_n> \varepsilon,\  x_n \uparrow \rR, \ s.t.\ L^{x_n}>2^3 (\rR-x_n)^\gamma \text{ i.o. }; \rR>\varepsilon)=0$ for all $\varepsilon>0$. Let $\varepsilon \downarrow 0$ to see that $\N_0( \exists x_n> 0,\  x_n \uparrow \rR, \ s.t.\ L^{x_n}>2^3 (\rR-x_n)^\gamma \text{ i.o. }; \rR>0)=0$. Since $\rR>0$, $\N_0$-a.e., we have $\N_0$-a.e. that $\exists \delta>0, \ s.t.\ \forall 0<\rR-x<\delta,\ L^x\leq 2^3 (\rR-x)^\gamma$. Use symmetry to conclude for $\lL$ and hence Theorem \ref{t1} holds if $\P_{\delta_0}$ is replaced with $\N_0$. The proof of Theorem \ref{t2} under $\N_0$ follows by similar arguments and Corollary \ref{t3} under $\N_0$ follows immediately from Theorem \ref{t2} under $\N_0$.
\end{proof}

\end{document}